\documentclass[12pt]{amsart}
\usepackage{geometry} % see geometry.pdf on how to lay out the page. There's lots.

\usepackage{amssymb}

\usepackage{amsmath}
\usepackage{amssymb,latexsym}
\usepackage{amsrefs}
\usepackage[draft=true]{hyperref}
\usepackage{cite}
\usepackage{a4wide}
\usepackage{amscd}
\usepackage{graphics}
\usepackage{color}
\usepackage[all]{xy}

\usepackage[all]{xy}
\newtheorem{theorem}{Theorem}[section]
\newtheorem*{theorem*}{Theorem}
\newtheorem{lemma}[theorem]{Lemma}
\newtheorem{corollary}[theorem]{Corollary}
\newtheorem{proposition}[theorem]{Proposition}
\theoremstyle{definition}
\newtheorem{definition}[theorem]{Definition}

\theoremstyle{remark}

\numberwithin{equation}{section}
\newcommand{\Z}{\ensuremath{\mathbb{Z}}}   %\Z integers Z

\newcommand{\Q}{\ensuremath{\mathbb{Q}}}

\newcommand{\Hom}{\operatorname{Hom}}

\newcommand{\Img}{\operatorname{Im}}

\newcommand{\Ext}{\operatorname{Ext}} %---- Ext yazmak icin

\newcommand{\Ker}{\operatorname{Ker}}
\newcommand{\Rad}{\operatorname{Rad}}
\newcommand{\Soc}{\operatorname{Soc}}
\newcommand{\Mod}{\operatorname{Mod-R}}
\newcommand{\coclosed}[2]{\ensuremath{\xymatrix@1{ #1 \, \ar@{^{(}->}[r]^-{cc} &
#2}}}
\newcommand{\cosmall}[3]{\ensuremath{\xymatrix@1{ #1 \,
\ar@{^{(}->}[r]^-{cs}_-{#2} & #3}}}

\newcommand{\ShortExactSequence}[5]{\ensuremath{\xymatrix@1{ 0 \ar[r] &  #1 \ar[r]^-{#2} & #3 \ar[r]^-{#4} &  #5  \ar[r] & 0 }}}
\newcommand{\LongExactSequence}[5]{\ensuremath{\xymatrix@1{ 0 \ar[r] & #5 \ar[r] &#1 \ar[r]^-{#2} & #3 \ar[r]^-{#4} &  #5  \ar[r] & 0 }}}
\newcommand{\AShortExactSequence}[6]{\ensuremath{\xymatrix@1{ #1: 0 \ar[r] &  #2 \ar[r]^-{#3} & #4 \ar[r]^-{#5} &  #6  \ar[r] & 0 }}}
\title{Co-Kasch Modules}
\author{Rafail Alizade}

\address{Ada University \\ School of Information Technologies and Eng.\\ Baku, Azerbaijan}
	
	\email{ralizada@ada.edu.az}

\author{Engin Büyükaşık}

\address{Izmir Institute of Technology \\ Department of Mathematics\\ 35430 \\ Urla, \.{I}zmir\\ Türkiye}
	
\email{enginbuyukasik@iyte.edu.tr}

\author{Yılmaz Durğun}

\address{Çukurova University \\ Department of Mathematics\\ Adana\\ Türkiye}
	
\email{enginbuyukasik@iyte.edu.tr}

\date{} % delete this line to display the current date

\begin{document}

%\tableofcontents

\begin{abstract}  In this paper  we  study the modules $M$ every simple subfactors of which is a  homomorphic image of $M$ and call them co-Kasch modules.   These modules are dual to Kasch  modules $M$ every simple subfactors of which can be embedded in $M$. We  show that a module is co-Kasch if and only if every simple module in $\sigma[M]$ is a homomorphic image of $M$. In particular, a projective right module $P$ is co-Kasch if and only if $P$ is a generator for $\sigma[P].$ 
 If $R$ is right max and right $H$-ring, then every right $R$-module is co-Kasch; and the converse is true for the rings whose simple right modules have locally artinian injective hulls.   For a right artinian ring $R$, we prove that: (1) every finitely generated right $R$-module is  co-Kasch if and only if every right $R$-module is a co-Kasch module if and only if $R$ is a right $H$-ring; and (2) every finitely generated projective right $R$-module is co-Kasch if and only if the Cartan matrix of $R$ is a diagonal matrix. For a Prüfer domain $R$, we prove that, every nonzero ideal of $R$ is co-Kasch if and only if $R$ is Dedekind. The structure of $\Z$-modules that are co-Kasch is completely characterized. \end{abstract}


\maketitle

\section{introduction}
A ring $R$ is said to be right Kasch if every simple right $R$-module embeds in $R$.  Commutative artinian rings and Quasi-Frobenious rings are  well-known examples of Kasch rings.  Kasch rings were generalized to modules by several authors. For example, in \cite{kasch}, a right $R$-module $M$ is called a right Kasch module if every simple subfactor of $M$ can be embedded into $M$. By a subfactor of $M$, we mean a submodule of a factor module of $M$.

Recall that, a ring $R$ is said to be a right $H$-ring if the injective hulls of nonisomorphic simple right $R$-modules are homologically independent, that is, $\Hom_R(E(S_1),\,E(S_2))=0$ for each nonisomorphic simple right $R$-modules $S_1$ and $S_2$ (see, \cite{sharp}). Commutative Noetherian rings, and commutative semiartinian rings are $H$-rings (see, \cite{camillohrings}). Thus, commutative perfect rings are $H$-rings, as left perfect rings are right semiartinian. Right artinian rings that are right $H$-ring are characterized in \cite{papp}. A ring $R$ is said to be  right max-ring if every nonzero right $R$-module has a maximal submodule.

In this paper, a notion dual to Kasch modules is proposed and studied. A natural dualization of Kasch modules is obtained as follows. We call a right module $M$ right co-Kasch module if every simple subfactor of $M$ is a homomorphic image of $M.$ Semisimple and free modules are trivial examples of co-Kasch modules.  More generally, for each right $R$-module $M$, the module $M\oplus R$ is a right co-Kasch module.

%A right $R$-module $N$ is called $M$-generated if there exists an epimorphism $M^{(I)} \to N$ for some index set $I$.  Following Wisbauer \cite{wisbauer}, $\sigma [M]$ denotes the full subcategory of $Mod-R$, whose objects are the submodules of $M$-generated modules. 

%A ring $R$ is called a right $H$-ring if the injective hulls of nonisomorphic simple right $R$-modules are homologically independent, that is, $\Hom_R(E(S_1),\,E(S_2))=0$ for each nonisomorphic simple right $R$-modules $S_1$ and $S_2$ (see, \cite{sharp}). Commutative Noetherian rings, and commutative semiartinian rings are $H$-rings (see, \cite{camillohrings}). Thus, commutative perfect rings are $H$-rings, as left perfect rings are right semiartinian. Right artinian rings that are right $H$-ring are characterized in \cite{papp}. A ring $R$ is called a right max-ring if every nonzero right $R$-module has a maximal submodule.

%Another dual notion of Kasch ring has been studied recently in \cite{dualkasch}. They call a ring $R$ right dual Kasch if every simple right $R$-module is a homomorphic image of an injective right $R$-module. We show that, $R$ is right dual Kasch if and only if the injective hull $E(R)$ of $R$ is a right co-Kaschmodule. 

The paper consists of three sections, and  is organized as follows. 

In section 2, we give some examples, closure properties and some characterizations of co-Kasch modules. %In \cite{kasch}, it is proved that $M$ is a right Kasch module if and only if every simple module in $\sigma[M]$  can be embedded in $M$. 
 We show that, $M$ is a right co-Kasch module if and only if every simple $S$ in $\sigma[M]$ is a homomorphic image of $M$.  A projective right $R$-module is co-Kasch if and only if $P$ is a generator for $\sigma[P].$ We  prove that, arbitrary direct sum of co-Kasch modules is again co-Kasch. On the other hand, direct summands or homomorphic image of co-Kasch modules need not be co-Kasch. It is shown that, if $M$ is a co-Kasch right module, then $\frac{M}{K}$ is co-Kasch for each $K \subseteq \Rad(M).$

 In section 3, mainly we focus on  investigation of  the rings  whose right $R$-modules are co-Kasch. Over a right $V$-ring, every right module is co-Kasch. If $R$ is a local ring, then every right $R$-module is co-Kasch if and only if $R$ is right perfect. We prove that, all right $R$-modules are co-Kasch if $R$ is right $H$-ring and right max-ring; and the converse statement is true if the injective hulls of simple right $R$-modules are locally artinian. If $R$ is right noetherian, then every right $R$-module is co-Kasch if and only if $R$ is right artinian right $H$-ring.   Note that a right $R$-module $M$ is locally artinian if every finitely generated submodule of $M$ is artinian; and $R$ is right $H$-ring if there is no nonzero homomorphism between the injective hulls of nonisomorphic simple right $R$-modules.\\
 We also deal with rings whose cyclic or finitely generated right modules are co-Kasch.  Over a commutative ring every cyclic module is co-Kasch. If $R$ is either local or  right $H$-ring, then every finitely generated right $R$-module is co-Kasch.
Over a right artinian ring $R$, all right modules are right co-Kasch if and only if every finitely generated right $R$-module is co-Kasch if and only if $\Ext_R(U,\,V)=0$ for each nonisomorphic simple right $R$-modules $U$ and $V.$  For a right artinian ring $R,$ we prove that every finitely generated projective right $R$-module is co-Kasch if and only if the Cartan matrix of $R$ is a diagonal matrix. If $E(R)$ is projective, then $E(R)$ is right co-Kasch if and only if $R$ is right self-injective.  
For a commutative Noetherian ring $R,$ we obtain that $E(R)$ is co-Kasch if and only if $R$ is a Kasch ring. \\
If  $M$ is a co-Kasch right $R$-module which is either right max-projective or a right semiartinian, then $M$ is retractable.
We obtain a characterization of right $H$-rings in terms of right Kasch modules. We prove that,  $R$ is a right $H$-ring if and only if $E(S)$ is a right Kasch module for each simple right $R$-module $S$. The rings whose right modules are Kasch are not completely known. Some sufficient conditions for such rings are given in \cite{fullykasch}. We give a necessary condition for such rings. Namely, we show that if every right $R$-module is right Kasch, then $R$ is a right $H$-ring and right semiartinian. 

In section 4,  co-Kasch modules are studied over commutative rings. For a Prüfer domain $R$, we prove that $R$ is Dedekind if and only if every nonzero ideal of $R$ is co-Kasch.  A characterization of co-Kasch modules over the ring of integers is obtained. Namely, we prove that, a nonzero torsion $\Z$-module $M$ is a co-Kasch module if and only if  $pM \neq M$ for each prime $p$ with $T_p(M) \neq 0$, where $T_p(M)$ is the $p$-primary component of the torsion submodule of $M$. If $M \neq  T(M)$, then  $M$ is a co-Kasch module if and only if $pM\neq M$ for each prime $p$. 

Throughout this paper, $R$ is a ring with unity and modules are unital right $R$-modules. As usual, we denote the radical, the socle and the injective hull of a module $M$ by $\Rad(M)$, $\Soc(M)$ and $E(M)$, respectively.  We write $J(R)$ for the Jacobson radical of the ring $R.$  We write $N \subseteq M$ if $N$ is a submodule of $M.$  We also write $M^{(I)}$ for the direct sum of $I$-copies of $M$.   $\Omega$ will stand for the set of maximal right ideals of a ring $R.$ The torsion submodule of a module $M$ over a commutative domain is denoted by $T(M)$. We refer to \cite{AF}, \cite{FC}, \cite{Lammodules} and \cite{wisbauer}  for all  undefined notions  in this paper.

\section{Co-Kasch modules}

As a dual notion to Kasch modules we investigate the following modules.

\begin{definition} Let $M$ be a right $R$-module.
\begin{enumerate}

\item[(a)] A module $N$ is  said to be a subfactor of $M$ if $N \subseteq \frac{M}{K}$ for some $K \subseteq M$. If, in addition,  $N$ is a simple module, $N$ is called a simple subfactor of $M.$

\item[(b)]  A right module $M$ is said to be a co-Kasch module if every simple subfactor of $M$ is a homomorphic image of $M.$ 
\end{enumerate}
\end{definition}

\begin{proposition} The following are equivalent for a right $R$-module $M.$

\begin{enumerate}

\item[(1)] $M$ is co-Kasch.

\item[(2)] $\Hom_R(M,\,S)\neq 0$ for each simple subfactor $S$ of $M.$

\item[(3)] If $\Hom_R(mR,\,S)\neq 0$ for some $m \in M$ and simple right $R$-module $S$, then $\Hom_R(M,\,S)\neq 0.$
\end{enumerate}

\end{proposition}
  
\begin{proof} $(1) \Rightarrow (2)$ is clear from the definition of co-Kasch module.

$(2) \Rightarrow (3)$ Suppose $\Hom_R(mR,\,S)\neq 0$ for some $m \in M$ and simple right $R$-module $S.$ Then $\frac{mR}{K} \cong S$ for some $K \subseteq mR.$ Thus $S$ is isomorphic to simple subfactor of $M,$ so $\Hom_R(M,\,S)\neq 0$  by $(2).$  

$(3) \Rightarrow (1)$ Let $ \frac{N}{K}$ be a simple subfactor of $M$, where $K \subseteq N \subseteq M.$ Let $n$ be an element of $N$ which is not in $K.$ Then $\Hom_R(nR,\,\frac{N}{K})\neq 0.$ Thus $\Hom_R(nR,\,\frac{N}{K})\neq 0$  by $(3).$ Hence $M$ is a co-Kasch module.
\end{proof}

%The characterization of right $H$-rings given in Proposition \ref{prop:H-ringcharacterization} motivates the following definition. 

In the following proposition we give some examples of modules that are  co-Kasch.

\begin{proposition}\label{prop:examplesHmodules} The following are hold.
\begin{enumerate}

\item[(1)]     The module $N \oplus R$ is  right co-Kasch module for every right $R$-module $N$.
  \item[(2)]   Semisimple right $R$-modules and free right $R$-modules  are co-Kasch modules.   
  \item[(3)] For every right $R$-module $M$, the modules $$ M_1=M \oplus (\oplus_{P \in \Omega}\frac{R}{P})\,\,\, \text{and}\,\,\,    M_2=M \oplus (\prod _{P \in \Omega}\frac{R}{P})$$ are right co-Kasch modules, where $\Omega$ is the set of all maximal right ideals of $R$.
  \end{enumerate}
  \end{proposition} 
  
  \begin{proof} $(1)$ Every simple right $R$-module is a homomorphic image of $R$. Thus $\Hom_R(N\oplus R,\,S)\neq 0$ for each simple right $R$-module $S.$ Hence $N\oplus R$ is a right co-Kasch module.
  
  $(2)$ Clearly semisimple modules are co-Kasch. On the other hand, free right modules are co-Kasch by $(1).$

  $(3)$ It is easy to see that every simple right $R$-module is a homomorphic image of $\oplus_{P \in \Omega}\frac{R}{P}$ and $\prod_{P \in \Omega}\frac{R}{P}$. Hence, $M_1$ and $M_2$ are right co-Kasch modules.
  %$(3)$ From their structure, we see that $\Hom_R(M_1, S)\neq 0$ and  $\Hom_R(M_2, S)\neq 0$ for each simple right $R$-module $S$. Hence $M_1$ and $M_2$ are right $co-Kasch.$
   \end{proof}

 A right $R$-module $N$ is said to be  $M$-generated if there exists an epimorphism $M^{(I)} \to N$ for some index set $I$.  Following Wisbauer \cite{wisbauer},
$\sigma [M]$ denotes the full subcategory of $Mod-R$, whose objects are the submodules of $M$-generated modules, that is,

$$\sigma[M]=\{N \in \Mod | N \subseteq \frac{M^{(I)}}{K}\,\,\text{for some index set}\,\,I \}.$$

%A module $N$ is said to be cogenerated by $M$ if there is a monomorphism $N \to M^I$ for some index set $I$.

In \cite{kasch}, it is shown that a right $R$-module $M$ is a Kasch module if and only if every simple module in $\sigma[M]$  can be embedded into $M$.
Now we shall prove a corresponding result for co-Kasch modules. First, we need the following lemma.

%\begin{proposition}\label{kaschmodules} The following properties are equivalent for the module $M$:
%\item[(1)] $M$ is a Kasch module;
%\item[(2)] Any simple module in $sigma[M]$  can be embedded in $M$; 
%\item[(3)] Any simple module in $\sigma[M]$ is cogenerated by $M$.
%\end{proposition}

\begin{lemma}\label{simplesubfactor} Every simple module in $\sigma[M]$ is isomorphic to a subfactor of $M.$
\end{lemma}

\begin{proof} Let $S$ be a simple module in $\sigma[M].$ Then there is a module $N$ such that $S \subseteq N$ and an epimorphism $f: M^{(I)} \to N$ for some index set $I.$  Let $S=aR$, where $a \in N.$ Then $a=f(b)$ for some $b \in M^{(I)}.$ Since $b \in M^{(F)}$  for some finite subset $F$ of $I,$ $S \subseteq f(M^{(F)}).$ Therefore we can replace $N$ by $f(M^{(F)})$. So we have an epimorphism $f: M_1 \oplus \ldots \oplus M_n \to N $ for some positive integer $n$, where each $M_i=M$ and $S \subseteq N.$  By induction on $n$ we will prove that $S$ is a subfactor of some $M_i,$ that is a subfactor of $M.$  The case $n=1$ is trivial. Let $n=k > 1$, and  assume that the statement is true for all $n<k.$ Denote $M_2 \oplus \ldots \oplus M_n$ by $L.$ If $f(M_1) \cap S \neq 0,$ then $S \subseteq f(M_1)$, therefore $S$ is a subfactor of $M_1.$ If $f(M_1) \cap S=0,$ then for the canonical epimorphism $\sigma :N \to \frac{N}{f(M_1)}=A$ we have $Ker(\sigma) \cap S= f(M_1) \cap S=0,$ so $S\cong (\sigma  f)(S) \subseteq A.$ On the other hand, for the projection $p: M_1 \oplus L \to L,$  we have $Ker(p)=M_1 \subseteq Ker (\sigma  f),$ therefore by Factor Theorem there is a homomorphism $g: L \to A$ such that $\sigma  f = g  p.$ Since $g  p = \sigma  f$ is an epimorphism $g$ is an epimorphism as well. So we have an epimorphism $g: M_2 \oplus \ldots \oplus M_n \to A$ and $S$ is isomorphic to a submodule of $A$. By the assumption of induction $S$ is a subfactor of  $M_i$ for some $i=2,\ldots, n,$ that is a subfactor of $M.$
\end{proof}

\begin{proposition}\label{co-KaschsigmaM} The following statements are equivalent for a right module $M$:

\begin{enumerate}
\item[(1)] $M$ is a right co-Kasch module,
\item[(2)] Every simple module in $\sigma[M]$ is a homomorphic image of $M,$
\item[(3)] Every simple module in $\sigma[M]$ is  generated by  $M$.
\end{enumerate} 
\end{proposition}

\begin{proof} $(1) \Rightarrow (2)$ Let $S$ be a simple module in $\sigma[M].$ Then $S$ is isomorphic to simple subfactor of $M$ by Lemma \ref{simplesubfactor}. Since $M$ is a co-Kasch module, $\Hom_R(M,\,S)\neq 0.$ That is, $S$ is a homomorphic image of $M$. This proves $(2).$

$(2) \Rightarrow (3)$ is clear. $(3) \Rightarrow (1)$ Let $S$ be a simple subfactor of $M.$ Then $S\in \sigma[M],$ and so there is a nonzero homomorphism $f: M \to S$ by $(3).$ Thus $\Hom_R(M, S)\neq 0,$ and so $M$ is a co-Kasch module.
 \end{proof}

Recall that,  given two right $R$-modules $P$ and $M$, the module $P$ is said to be $M$-projective if for any  epimorphism $g: M \to N$ and homomorphism $f: P \to N$, there exists $h: P \to M$ such that $f=gh.$

\begin{proposition}  Let  $P$ be an $M$-projective right $R$-module. Then $P$ is a co-Kasch module if and only if $P$ is a generator for the category $\sigma[M].$
\end{proposition}

\begin{proof} By \cite[18.5]{wisbauer}, an $M$-projective right $R$-module $P$ is a generator for $\sigma[M]$ if and only if $\Hom_R(P, S)\neq 0$ for each simple module in $\sigma[M].$ Hence the conclusion follows by  Proposition \ref{co-KaschsigmaM}. 
\end{proof}

\begin{proposition}\label{prop:diretsumofco-Kasch} If $\{M_i \}_{i\in I}$ is a family of co-Kasch modules, for an index set $I$, then $\oplus_{i \in I}M_i$ is a co-Kasch module.
\end{proposition}

\begin{proof} Let $S$ be a simple subfactor of $\oplus_{i \in I}M_i$. Then $S \subseteq \frac{\oplus_{i \in I}M_i}{K}$ for some $K \subseteq \oplus_{i \in I}M_i.$ Set $N=\frac{\oplus_{i \in I}M_i}{K}$ and let $f: \oplus_{i \in I}M_i \to N$ be the canonical epimorphism.  Now, exactly the same arguments as in the proof of Lemma \ref{simplesubfactor} shows that $S$ is isomorphic to a subfactor of $M_i$ for some $i \in I$. Since $M_i$ is co-Kasch module, $\Hom_R(M_i,\,S)\neq 0$, and so $\Hom_R(\oplus_{i \in I}M_i,\,S)\neq 0.$ Thus the module $\oplus_{i \in I}M_i$ is co-Kasch.
\end{proof}

Co-Kasch modules are not closed under direct summands or under factor modules. For example, the $\Z$-module $\Z \oplus \Q$ is co-Kasch by Proposition \ref{prop:examplesHmodules}(1). On the other hand, $\Q$ is not a co-Kasch $\Z$-module since it has no any simple factor.

\begin{proposition} Let $M$ be right co-Kash module. Then $\frac{M}{K}$ is co-Kasch for every $K \subseteq \Rad(M).$
\end{proposition}

\begin{proof} Let $S$ be a simple subfactor of $\frac{M}{K}$. Clearly, $S$ is a simple subfactor of $M$, as well. Then $\Hom_R(M, S)\neq 0,$ so there is a nonzero epimorphism $f:M \to S.$ Since $S$ is a simple module, $\Ker(f)$ is a maximal submodule of $M$. Thus $\Rad(M) \subseteq \Ker(f)$ and so, by factor theorem, there is a nonzero homomorphism $g: \frac{M}{K} \to S$ such that $g\pi=f,$ where $\pi: M \to \frac{M}{K}$ is the canonical epimorphism. Therefore $\Hom_R(\frac{M}{K}, S)\neq 0,$ and so $\frac{M}{K}$ is a co-Kasch module.
\end{proof}

\section{Rings whose modules are co-Kasch}

 In this section, we shall deal with the rings whose right modules are co-Kasch.  
   
 \begin{proposition}\label{prop:radM} The following are hold.
 \begin{enumerate}
 \item[(1)] If $M$ is a nonzero right co-Kasch module, then $\Rad(M) \neq M$. 
 \item[(2)] If $R$ is a local ring, then a nonzero right $R$-module $M$ is co-Kasch if and only if $\Rad(M) \neq M.$
 \item[(3)] If every right $R$-module is co-Kasch, then $R$ is a right max-ring.
 \end{enumerate}
 \end{proposition}
 
  \begin{proof}$(1)$  Let $N$ be a nonzero right $R$-module. Let $0 \neq x \in N.$ Then, $xR$ has a maximal submodule say $K$. Then $S=\frac{xR}{K}$ is a simple subfactor of $N$. By the hypothesis, $N$ is a co-Kasch module and so $\Hom_R(N,\,S)\neq 0.$ Therefore $N$ has a maximal submodule, and so $\Rad(N)\neq N.$

$(2)$ Necessity follows from $(1).$ Suppose that $\Rad(M) \neq M.$ Then $M$ has a simple factor. Since $R$ is local,  $\frac{R}{J(R)}$ is the  unique simple right $R$-module up to isomorphism. Then $\Hom_R(M,\, \frac{R}{J(R)})\neq 0.$ Thus $M$ is co-Kasch.

$(3)$ Let $M$ be a nonzero right $R$-module. If $M$ is co-Kasch then $\Rad(M) \neq M$ by $(1).$ Thus, every nonzero right $R$-module has a maximal submodule. Hence $R$ is a right max-ring.

\end{proof}
 
 %Now we shall characterize the rings all of whose right modules are co-Kasch. We recall the following definitions. 

%\textbf {burada tanımları yazmayı unutmalayalım}
%\begin{definition}

%\begin{enumerate}

%\item[(a)] A ring $R$ is said to be a right $H$ ring if the injective hulls of nonisomorphic simple right $R$-modules are homologically independent, that is, $Hom(E(S_1),\,E(S_2))=0$ for each nonisomorphic simple right $R$-modules $S_1$ and $S_2$. Commutative Noetherian rings, and commutative semiartinian rings are $H$-rings (see, \cite{camillohrings}). Thus, commutative perfect rings are $H$-ring, as left perfect rings are right semiartinian. Right artinian rings that are right $H$-ring are characterized in \cite{papp}. \item[(b)] A ring $R$ is said to be right max if every nonzero right $R$-module has a maximal submodule. 
%\end{enumerate}
%\end{definition}

In the following theorem we give some sufficient conditions for the rings whose right modules are co-Kasch.

\begin{theorem}\label{prop:allmodulesareHmodules} Consider the following statements for a ring $R.$   
\begin{enumerate}
\item[(1)] $R$ is a right $H$-ring and a right max-ring.
\item[(2)] Every right $R$-module is a co-Kasch module.  
%\item[(2')] For each simple right $R$-module $S$, every submodule of $E(S)$ is an co-Kaschmodule
\item[(3)] Right co-Kasch modules are closed under factor modules.
\item[(4)] Right co-Kasch modules are closed under submodules.
\item[(5)] Right co-Kasch modules are closed under direct summands.

\end{enumerate}

Then $(1) \Rightarrow (2) \iff (3) \iff (4) \iff (5)$
\end{theorem}

\begin{proof} 
%$(2) \Rightarrow (1)$ Suppose every right $R$-module is an co-Kaschmodule. Then $R$ is a right $H$-ring by Proposition \ref{prop:H-ringcharacterization}. Since every nonzero right module has a simple subfactor, every nonzero module has a simple factor module. Thus $R$ is right max-ring. This proves $(2).$

$(1) \Rightarrow (2)$ Let $M$ be a right $R$-module and $S=\frac{N}{K}$ be a simple subfactor of $M$. Then there is a nonzero homomorphism $g: \frac{M}{K} \to E(S)$ such that the following diagram is commutative.  
$$\xymatrix{0 \ar[r]  &  \frac{N}{K} \ar[d]_{f} \ar[r]^{i}& \frac{M}{K} \ar@{.>}[dl]^{g} &\ar[l]_{\pi}M\\
 & E(S) & &
} $$

Thus $g\pi :M \to E(S)$ is a nonzero homomorphism, where $\pi: M \to \frac{M}{K}$ is the canonical epimorphism. Then for $L=\Ker(g\pi)$, we have $\frac{M}{L} \cong Im(g \pi ) \subseteq E(S)$. Since $R$ is a right max-ring, $\frac{M}{L}$ has a simple factor say $X$. Then $X\cong S$ by the right $H$-ring assumption. Hence, $M$ has a simple factor isomorphic to $S$, and so $M$ is a right co-Kasch module. This proves (2).

%$(2) \Rightarrow (1)$  Assume that $R$ is not a right $H$-ring. Then there are two nonisomorphic simple right $R$-modules $S_1$ and $S_2$ such that $\Hom_R(E(S_1),\,E(S_2)) =0.$ Let $f: E(S_1) \to E(S_2)$ be a nonzero homomorphism. Then $\frac{E(S_1)}{\Ker(f)} \cong Im(f),$ and so there are submodules $A \subseteq E(S_1)$ and $A_1 \subseteq A$ such that $\frac{A}{A_1} \cong S_2.$ By $(2)$, $A$ is a co-Kasch module, and as $S_1 \subseteq A$ there is a submodule $B_1 \subseteq A$ such that $\frac{A}{B_1} \cong S_1.$ Now $A_1$ and $B_1$ are distinct maximal submodules of $A$ and so $A=A_1 +B_1$ and $$S_1\cong \frac{A}{B_1} \cong \frac{A_1 +B_1}{B_1} \cong \frac{A_1}{A_1\cap B_1} .$$ Consider the right $R$-module $M=\frac{A}{A_1 \cap B_1}.$ Then we have a short exact sequence $$0\to S_1 \to M \to S_2\to 0$$ from which we see that $S_1$ embeds in $M$ and $\Hom_R(M,\,S_1)=0$. Thus $M$ is not an co-Kasch module. This contradicts $(2).$ Hence $R$ must be a right $H$-ring.

%To prove that $R$ is  right max-ring, consider a nonzero right $R$-module $N$. Then, as $N$ is nonzero $N$ has a simple subfactor, say $S.$ By $(2)$, $N$ is a co-Kasch  %module and so $\Hom_R(N,\,S)\neq 0.$ Therefore $N$ has a maximal submodule. This implies that $R$ is a right max-ring. This completes the proof.

$(2) \Rightarrow (3)  \Rightarrow (5)$ are clear. 
$(5) \Rightarrow (2)$  Let $M$ be a right $R$-module. Then $M\oplus R$ is a right co-Kasch module  by Proposition \ref{prop:examplesHmodules}. Then $M$ is a right co-Kasch module by $(5)$. Therefore, every right $R$-module is a co-Kasch module.

$(2) \Rightarrow (4) \Rightarrow (5)$ are clear.
\end{proof}

 A ring $R$ is said to be right co-noetherian if the injective hull of each simple right $R$-module is artinian. Right $V$-rings and $QF$-rings are trivial examples of right co-noetherian rings. Commutative noetherian rings are co-noetherian (see \cite{matlis}). A ring $R$ is said to  satisfy the $(\diamond)$ property if $E(S)$ is locally artinian for each simple right $R$-module $S.$ Right co-noetherian and right artinian rings satisfy $(\diamond).$ For characterizations of certain rings and  algebras that satisfy the  $(\diamond)$ property we refer to \cite{diamond1}, \cite{diamond}, \cite{diamond2}, \cite{diamond3}.
 
We do not know whether the statement $(1)$  in Theorem \ref{prop:allmodulesareHmodules} is necessary for the implication $(2)$, in general.  For certain rings including the rings satisfy $(\diamond)$, right noetherian rings and local rings the assumption that all right modules are co-Kasch implies  the ring is right $H$-ring and right max-ring.
We begin with the following.
 \begin{theorem}\label{diamond}The following are equivalent for a  ring $R$ that satisfy $(\diamond).$
\begin{enumerate}
\item[(1)] $R$ is a right $H$-ring and right max-ring.
\item[(2)] Every right $R$-module is co-Kasch.
\item[(3)]  Every subfactor of $E(S)$ is  co-Kasch, for each simple right $R$-module $S.$

\end{enumerate}
\end{theorem}

\begin{proof} $(1) \Rightarrow (2)$ By Theorem \ref{prop:allmodulesareHmodules}, and $(2) \Rightarrow (3)$ is clear. 

$(3) \Rightarrow (1)$ First, let us prove that $R$ is a right $H$-ring. Assume the contrary. Then there are two nonisomorphic simple right $R$-modules $S_1$ and $S_2$ 
with $\Hom_R(E(S_1),\,E(S_2)) \neq 0.$ Let $f: E(S_1) \to E(S_2)$ be a nonzero homomorphism. Since $S_2$ is essential in $E(S_2),$ $S_2$ is essential in $\Img(f).$ 
Denote $\Ker(f)$ by $K_1.$ Let $A_1=f^{-1} (S_2).$ Since $S_2$ is cyclic, without loss of generality, we may assume that $A_1$ is a cyclic module as well.  Since $S_1 \ncong  S_2$ and $S_2$ is essential in $E(S_2),$ $f(S_1)=0,$ so $S_1 \subseteq K_1 \subseteq A_1.$ By (2),  $K_1$ is a co-Kasch module, therefore there is a nonzero homomorphism $g_1: K_1 \to S_1,$ which is clearly an epimorphism. Taking a pushout diagram we will have the following commutative diagram with exact rows and columns.

$$\xymatrix{  & 0 &  &  & \\  
E' :0 \ar[r]  & S_1 \ar[u] \ar[r]_{h}  & B_1  \ar[r]^{k}  \ar@{.>}@/_/[l]_{s} & S_2  \ar[r] \ar@{=}[d] & 0\\  
E_1 :0 \ar[r]  & K_1 \ar[u]^{g_1} \ar[r]  & A_1  \ar[u] \ar[r]   & S_2   \ar[r] & 0\\
 & K_2 \ar[u]^{i_1} &  &  & \\
 & 0 \ar[u] &  &  &
}$$
where $K_2= \Ker (g_1)$ and $i_1 : K_2 \to K_1$ is the inclusion map. 
Since $B_1$ is a co-Kasch module there is a nonzero homomorphism $s:B_1 \to S_1.$ If $\Ker(k)=\Img(h) \subseteq \Ker (s),$ then there is an epimorphism $S_2 \to S_1$ by the factor theorem, contradiction. So $\Img(h)   \nsubseteq   \Ker (s),$ therefore $sh: S_1 \to S_1$ is an isomorphism, that is $E'$ is splitting. If $K_1=S_1,$ then $g_1=1_{S_1},\,\,B_1=A_1,$ so $S_1$ is a direct summand of $A_1$, a contradiction. Let $K_1 \neq S_1.$ From the long exact sequence 

$$\xymatrix@1{ \cdots \ar[r] & \Ext_R( S_2,\, K_2) \ar[r]^{i_{1_*} } & \Ext_R( S_2,\, K_1) \ar[r]^{g_{1_*} } & \Ext_R( S_2,\, S_1)   \ar[r]  & \cdots }$$
induced by the short exact sequence 
$\xymatrix@1{E'': 0 \ar[r]  & K_2 \ar[r]^{i_1}  & K_1 \ar[r]^{g_1}  & S_1  \ar[r] & 0},$
we obtain that $E_1 \in \Ker (g_{1_*})= \Img (i_{1_*}),$ therefore $E_1= i_{1_*}(E_2)$ for some $$\xymatrix@1{E_2 : 0 \ar[r]  & K_2 \ar[r]  & A_2 \ar[r]  & S_2  \ar[r] & 0.}$$
If $S_1 \nsubseteq K_2$ then $E''$ is splitting, so $K_1$ contains a direct summand $X$ isomorphic to $S_1.$ Then either $X\cap S_1=0$ or $X=S_1.$ Both cases contradicts with the fact that $S_1$ is essential in $K_1,$ and so $S_1 \subseteq K_2.$ Since $K_2$ is a co-Kasch module, there is an epimorphism $g_2: K_2 \to S_1.$ In  a similar way we will have the following  commutative diagram with exact rows and columns, first row of which is splitting
 
$$\xymatrix{  & 0 &  &  & \\  
0 \ar[r]  & S_1 \ar[u] \ar[r]  & B_2  \ar[r] & S_2  \ar[r] \ar@{=}[d] & 0\\  
0 \ar[r]  & K_2 \ar[u] \ar[r]  & A_2  \ar[u] \ar[r]   & S_2   \ar[r] & 0\\
 & K_3 \ar[u] &  &  & \\
 & 0 \ar[u] &  &  &
}$$
Continuing in this way we will have a descending chain $$A_1 \supseteq K_1 \supseteq K_2 \supseteq \cdots \supseteq S_1 $$and corresponding exact diagrams. Since $E(S_1) $   is locally artinian and $A_1$ is cyclic, $A_1$ artinian. Thus $K_n=S_1$ for some positive integer $n,$ and so $S_1$ is a direct summand in $A_n,$ a contradiction. Therefore, $R$ is a right $H$-ring.

Now let us prove that $R$ is right max-ring as well. For this, suppose the contrary that there is a nonzero right $R$-module $M$ such that $\Rad(M)=M.$ Let $\mathcal{S}$ be the set of representatives of nonisomorphic simple right $R$-modules. Then $E=\oplus_{S \in \mathcal{S}}E(S)$ is an injective cogenerator for the category of right $R$-modules. 
Thus there is a monomorphism $f: M \to E^I$ for some index set $I.$ Since $f$ is nonzero, there is a nonzero homomorphim $g:M \to E(S)$   for some $S \in \mathcal{S}.$ 
Then $\frac{M}{\Ker(g)} \cong g(M)\subseteq E(S)$ implies that, $\Rad(g(M))=g(M)$. Thus, $g(M)$ is not a co-Kasch module. This contradicts the assumption $(2)$.  
Therefore, $\Rad (M)\neq M$ for every right $R$-module $M$. Hence $R$ is a right max-ring. 
\end{proof}

A ring $R$ is said to be  right $V$-ring if every simple right $R$-module is injective.

\begin{corollary} Over a right $V$-ring, every right $R$-module is a co-Kasch module.
\end{corollary}

\begin{proof}Let $R$ be a right $V$-ring. Every simple right $R$-module is injective, so $R$ is a right $H$-ring. On the other hand, $\Rad (M)=0$, for every right $R$-module $M$ (see \cite{Lammodules}). Thus $R$ is a right max-ring. Now the proof follows by Theorem \ref{diamond}.

\end{proof}

\begin{corollary}Let $R$ be a right noetherian ring. The following are equivalent.

\begin{enumerate}
\item[(1)] Every right $R$-module is co-Kasch.
\item[(2)] $R$ is right artinian right $H$-ring.
\end{enumerate}
\end{corollary}

\begin{proof} $(1) \Rightarrow (2)$ Suppose every right $R$-module is co-Kasch. Then $R$ is right max by Proposition \ref{prop:radM}. Being right noetherian and right max, $R$ is right artinian by \cite{Lamring}. Then $R$ is right $H$-ring by Theorem \ref{diamond}. This proves $(2).$

$(2) \Rightarrow (1)$ By Theorem \ref{prop:allmodulesareHmodules}.
\end{proof}

Local rings are right  $H$-rings, because a local ring has a unique simple right module up to isomorphism. Semilocal right max-rings are perfect by \cite{Lamring}, and a local ring $R$ is right perfect if and only if $R$ is right max-ring. Hence the following is clear by Proposition \ref{prop:radM}(2).

\begin{corollary}  If $R$ is a local ring, then every right $R$-module is a co-Kasch module if and only if $R$ is right perfect.
\end{corollary}

There are local rings that are right perfect which do not satisfy the $(\diamond)$ property (see, \cite[Example 5.2]{diamond}).

\begin{proposition} The following statements are equivalent for a right artinian ring $R$.

\begin{enumerate}
\item[(1)] Every  right $R$-module is a co-Kasch module.
\item[(2)] Every finitely generated right $R$-module is a co-Kasch module.
%\item[(3)] Every cyclic right $R$-module is a co-Kasch module.
\item[(3)] $\Ext_R(S_1,\,S_2)=0$ for each nonisomorphic simple right $R$-modules $S_1,\,S_2.$
\item[(4)] $R$ is a right $H$-ring.
\end{enumerate}
\end{proposition}

\begin{proof} $(1) \Rightarrow (2)$ is clear.  $(2) \Rightarrow (4)$ Note that right artinian rings satisfy the $(\diamond)$ property. With same notation as in the  the proof of Theorem \ref{prop:allmodulesareHmodules} $(2) \Rightarrow (1)$ we have $A_1$ is cyclic. So that by the right artinian assumption every submodule of $A_1$ is finitely generated. Hence similar arguments shows that $R$ is right $H$-ring.

%Suppose the contrary, that $\Ext_R(S_1,\,S_2) \neq 0$. Then there is a module $A$ containing $S_1$ such that $S_1$ is essential in $A$ and $\frac{A}{S_1} \cong S_2$. Since $A$ is cyclic, it is a co-Kasch module by $(3)$. But $S_1$ is not a homomorphic image  of $A$? $S_1$?. Contradiction!  Thus we must have $\Ext_R(S_1,\,S_2)=0$.

$(3) \Leftrightarrow (4)$ By \cite[Theorem 9]{papp} a ring $R$ is a right $H$-ring if and only if $\Ext_R(S_1,\,S_2)=0$ for each nonisomorphic simple right $R$-modules $S_1,\,S_2.$

$(4) \Rightarrow (1)$ Right artinian rings are right max. Thus $(1)$ follows by Proposition \ref{prop:H-ringcharacterization}.
\end{proof}

A nonzero idempotent $e \in R$ is said to be primitive  if $eR$ is indecomposable as a right $R$-module (see, \cite[Proposition 21.8]{Lamring}). Such modules $eR$ are called principal indecomposable modules. Principal indecomposable right modules are local, and so each principal indecomposable module $eR$ has a unique maximal submodule which is $eJ,$ where $J$ denotes the Jacobson radical of $R.$  Every right artinian ring can be written as a direct sum of principal indecomposable right $R$-modules.

For a right artinian ring $R$, let $e_1R,\ldots, e_rR$ represent a complete set of isomorphism classes of principal indecomposable right $R$-modules, and let $V_j =\frac{e_jR}{e_jR}$ so that $V_1,\ldots , V_r$ represent a complete set of isomorphism classes of simple right $R$-modules \cite[Theorem 25.3 (1)]{Lamring}. Let $c_{ij}   \geq 0$ be the number of composition factors of $e_iR$ which are isomorphic to $V_j$. The matrix $$C = (c_{ij}) \in \mathbb{M}_r(\Z)$$
is called the (right) Cartan matrix of $R$. Note that the diagonal elements $c_{ii}$ are $\geq 1$, and the sum of the {\it i}th row of $C$ is just the (composition) length of $e_iR.$

\begin{proposition} Let $R$ be a right artinian ring.  The following statements are equivalent.
\begin{enumerate}

\item[(1)] Every finitely generated projective right $R$-module is a co-Kasch module;
\item[(2)] Every principal indecomposable right $R$-module is a co-Kasch module; 
\item[(3)] The Cartan matrix of $R$ is a diagonal matrix.
\end{enumerate}
\end{proposition}

\begin{proof} $(1) \Rightarrow (2)$ is clear. 

$(2) \Leftrightarrow (3)$ Let $e_1R,\cdots , e_s R$  be the principal indecomposable right $R$-modules. For $i\in \{1,\cdots ,n \}$ 
let $$0=A_0\subseteq A_1 \subseteq \cdots \subseteq A_{n_i}=e_iR$$be a composition series of $e_iR.$ Set $V_t =\frac{A_{t}}{A_{t-1}}$ for $t=1,\cdots n_i.$ Since $e_iR$ is a local right $R$-module, $\frac{e_iR}{e_iJ}$ is the unique simple factor of $e_iR.$  Now it is easy to see that  $V_{t} \cong \frac{e_iR}{e_iJ}$ for each $t=1,\cdots n_i$ if and only if $c_{ii}=cl(e_iR),$ and $c_{ij}=0$ for each $j\neq i.$ Hence the proof follows.

$(2) \Rightarrow (1)$ Over a right artinian ring every finitely generated projective right $R$-module $P$ is a direct sum of principal indecomposable right $R$-modules (see, \cite[Theorem 25.3(2)]{Lamring}). Now, (1) follows by $(2)$ and Proposition \ref{prop:diretsumofco-Kasch}.
\end{proof}

%Thus we have the following corollary:

%\begin{corollary}If $R$ is right max and right $H$-ring, then $E(S)$ is an $H$-module for every simple right $R$-module $S$.
%\end{corollary}

%The assumption that $R$ is right max, is not superflous. For example, $\Z$ is an $H$-ring, but $\Z_{p ^{\infty}}$ is not an $H$-module, because it has no maximal %submodules.

%\begin{corollary}  A commutative semilocal ring $R$ is perfect if and only if  and every $R$-module is a co-Kasch module.
%\end{corollary}

%\begin{proof}  Commutative perfect rings are $H$-rings. Hence the proof follows by Proposition \ref{prop:allmodulesareHmodules}. 
%\end{proof}

%Every right Noetherian right max-ring is right artinian, hence the following corollary holds.

%\begin{corollary} Let $R$ be a right Noetherian ring. Every  right  $R$-module is a co-Kasch module if and only if $R$ is a right artinian, right $H$-ring.
%\end{corollary}

%\begin{proof} Let us prove the necessity. 

%\end{proof}

At this point, it is natural to ask what are the rings whose cyclic  or finitely generated right modules are co-Kasch modules.  It is easy to see that local rings have this property by Proposition \ref{prop:radM}(2).

\begin{proposition} If $R$ is commutative, then every cyclic $R$-module is a co-Kasch module.
\end{proposition}

\begin{proof} Let $C\cong \frac{R}{I}$ be a cyclic $R$-module and $X$ be a simple subfactor of $\frac{R}{I}$. Then $X \cong \frac{A}{J} \subseteq \frac{R}{J}$ for some ideal $J$ with $I\subseteq J$. Let $P=ann_R(\frac{A}{J})$. Note that $\Hom_R( \frac{R}{I},\, \frac{R}{P})\neq 0.$ Then, as $X$ is a  simple $R$-module, $P$ is a maximal ideal of $R$ with $J \subseteq P.$ Then $\frac{A}{J} \cong \frac{R}{P}$ and so $\frac{A}{J}$ is isomorphic to a simple factor of $\frac{R}{I}$. Hence $\frac{R}{I}$ is a  co-Kasch module. 
\end{proof}

\begin{proposition} If $R$ is a right $H$-ring, then every finitely generated right $R$-module is a co-Kasch module.
\end{proposition}

\begin{proof} 
Suppose that $R$ is a right $H$-ring, and $M$ a finitely generated right $R$-module.  Let $S=\frac{A}{B}\subseteq \frac{M}{B}$ be a simple subfactor of $M$.  There is a nonzero homomorphism $f:\frac{M}{B} \to E(S)$.  Then $\frac{M}{\Ker (f)}$ is isomorphic to a nonzero submodule of $E(S)$.   Since $\frac{M}{\Ker (f)}$ is finitely generated and nonzero,  there is a maximal submodule $K$ of $M$ such that $\frac{M}{K}\cong U$, where  $U$ is a simple subfactor of $E(S)$. By the hypothesis, $R$ is a right $H$-ring, thus $U \cong S.$  Therefore every simple subfactor of $M$ is isomorphic to  a simple factor of $M$. Hence $M$ is a co-Kasch module.
\end{proof}

%\begin{corollary} Let $R$ be a $QF$ ring. The following statements are equivalent.
%\end{corollary}
%\begin{proof} 
%\end{proof}

A ring $R$ is said to be right dual Kasch if every simple right $R$-module is a homomorphic image of an injective right $R$-module. Dual Kasch rings are investigated and studied in \cite{dualkasch}. 

\begin{proposition}\label{prop:E(R) H-module} $E(R)$ is a right co-Kasch module if and only if $R$ is right dual Kasch.
\end{proposition}

\begin{proof} Suppose $E(R)$ is a right co-Kasch module.  Since $R \subseteq E(R)$, every simple right $R$-module is a subfactor of $E(R).$  Thus every simple right $R$-module is a homomorphic image of $E(R)$, and so  $R$ is right dual Kasch. This proves the necessity. 
Sufficiency follows by \cite[Theorem 2.1(2)]{dualkasch}.
\end{proof}

\begin{corollary} If $R$ is a right max and right $H$-ring, then $R$ is right dual Kasch.
\end{corollary}

A commutative ring is said to be classical  if every element is either a zero-divisor or a unit.  
  
\begin{corollary} Let $R$ be a commutative Noetherian ring. The following statements are equivalent.
\begin{enumerate}
\item[(1)] $E(R)$ is a co-Kasch module.
\item[(2)] $R$ is a classical ring.
\item[(3)] $R$ is a dual Kasch ring.
\item[(4)] $R$ is a Kasch ring.
\end{enumerate}
\end{corollary}

\begin{proof} $(1) \Leftrightarrow (3)$ By Proposition \ref{prop:E(R) H-module}. $(2) \Leftrightarrow (3) \Leftrightarrow (4)$ By \cite[Theorem 3.4.]{dualkasch}.
\end{proof}

\begin{proposition} If $E(R)$ is projective, then $E(R)$ is a right co-Kasch module if and only if $R$ is right self-injective.
\end{proposition}

\begin{proof} Suppose that $E(R)$ is a right co-Kasch module. Then every simple right $R$-module is a homomorphic image of $E(R)$. Then projectivity of $E(R)$ implies that, $E(R)$ is a  generator for Mod-$R$. Thus $E(R)^n \cong K \oplus R$ for some $n \in \Z^+$. Therefore $R$ is right self-injective.  Sufficiency is clear.
\end{proof}

Recall that, a right $R$-module $M$ is said to be retractable if $\Hom_R (M,\,N)\neq 0$ for each nonzero submodule $N$ of $M$. A right $R$-module $M$ is said to be max-projective if for every maximal right ideal $I$ of $R$, any homomorphism $f: M \to R/I$ can be lifted to a homomorphism $g:M \to R$ (see, \cite{maxprojective}).

\begin{lemma} Let $M$ be a right max-projective module. If $M$ is a co-Kasch module, then  $M$ is retractable.
\end{lemma}

\begin{proof} Let $N$ be a nonzero submodule of $M$, and $0 \neq x \in N$. Let $K$ be a maximal submodule of $xR$. Since $M$ is a co-Kasch module and $\frac{xR}{K}$ is a simple subfactor of $M,$ $\Hom_R(M, \frac{xR}{K} )\neq 0$. Let $f: M \to \frac{xR}{K}$ be a nonzero homomorphism. By the max-projectivity assumption on $M$, there is a  homomorphism $g: M \to xR$ such that $\pi g=f$, where $\pi :xR \to \frac{xR}{K} $ is the canonical epimorphism. Then $gi: M \to N$ is a nonzero homomorphism, where $i:xR \to N$ is the inclusion map. Therefore, $M$ is a retractable right $R$-module, as desired.
\end{proof}

\begin{proposition} Let $M$ be a  right semiartinian co-Kasch module. Then $\Hom_R (M,\,\frac{N}{K})\neq 0$ for each submodules $K \subsetneq N \subsetneq M.$  In particular, $M$ is retractable.
\end{proposition}

Now we shall mention several results about Kasch modules.
We observe the following characterization of right $H$-rings in terms of Kasch modules.

\begin{proposition}\label{prop:H-ringcharacterization} The following statements are equivalent for a ring $R$: 
\begin{enumerate} 
\item[(1)] $R$ is a right $H$-ring; 
\item[(2)] every simple subfactor of $E(S)$ is isomorphic to $S$ for each simple right $R$-module $S;$
\item[(3)] for every simple right $R$-module $S$, each submodule of $E(S)$ is a Kasch module;
\item[(4)]  $E(S)$ is a Kasch module for each simple right $R$-module $S.$

\end{enumerate}
\end{proposition}

\begin{proof}$(1) \Rightarrow (2)$ Suppose $A$ is a right $H$-ring. Let $S$ be a simple right $R$-module and $E(S)$ be its injective hull.  Let $U=\frac{X}{Y}$ be a simple subfactor of $E(S)$. We need to show that, $S\cong U.$ Assume the contrary that $U$ and $S$ are not isomorphic. Let $f: \frac{X}{Y} \to \frac{E(S)}{Y}$  and $g: \frac{X}{Y} \to E(\frac{X}{Y})$ be the inclusion homomorphisms. Then by injectivity of $E(\frac{X}{Y})$, there is a (nonzero) homomorphism $h: \frac{E(S)}{Y} \to E(\frac{X}{Y})$. Therefore for the natural epimorhism $\pi : E(S) \to  \frac{E(S)}{Y}$, we obtain $0 \neq h\pi \in \Hom(E(S),\, E(\frac{X}{Y})).$ This contradicts  the fact that, $R$ is right $H$-ring. 

$(2) \Rightarrow (1):$ Let $S_1$ and $S_2$ be nonisomorphic simple right $R$-modules. Suppose that there is a nonzero homomorphism $f: E(S_1) \to E(S_2)$. Then $\frac{E(S_1)}{\Ker (f)} \cong f(E(S_2))$. Thus, there is a simple factor $\frac{X}{\Ker (f)} \cong S_2$. This contradicts the assumption that simple subfactors of $E(S_1)$ are isomorphic to $S_1$. Thus, we must have $f=0,$ and so $R$ is a right $H$-ring.

$(2) \Rightarrow (3)$ Let $K$ be a nonzero submodule of $E(S)$ and $U$ be a simple subfactor of $K$. Then $U$ is also a subfactor of $E(S)$, and so $U$ is isomorphic to $S$ by $(2).$ Since $Soc(E(S))=S$ and $K$ is nonzero, $Soc(K)=S.$ Therefore $U$ embeds in $K$, and so $K$ is a Kasch module.

$(3) \Rightarrow (2)$ and $(3) \Rightarrow (4)$  are clear.
\end{proof}

The authors in \cite{kasch}, ask about the rings all whose right modules are Kasch. To the best of our knowledge, this question is not solved completely. In \cite{fullykasch}, the authors gives some sufficient conditions for the rings having this property. We observed the following necessary conditions for the rings whose right modules are Kasch.

\begin{proposition}\label{prop:allmodulesKasch} If every right $R$-module is a Kasch module, then $R$ is a right $H$-ring and right semiartinian.
\end{proposition}

\begin{proof} Suppose every right $R$-module is  Kasch. Then $R$ is a right $H$-ring by  Theorem \ref{prop:H-ringcharacterization}. Let $M$ be a nonzero right $R$-module and $0\neq m \in M.$ Then $mR$ has a maximal submodule $K$, and $\frac{mR}{K}$ is a simple subfactor of $M$. Therefore $\frac{mR}{K}$ embeds in $M$, because $M$ is a Kasch module and so  $\Soc(M) \neq 0.$ Thus the assumption implies that every nonzero right $R$-module contains a simple submodule. Hence $R$ is right semiartinian.

\end{proof}

The converse of Proposition \ref{prop:allmodulesKasch} is not true in general. That is, there are right semiartinian right $H$-rings that admit right modules  that are not right Kasch. For example the ring given in \cite[Remarks 3.14 (3)]{kasch}, is a semartinian  V-ring (hence an H-ring) which has a module that is not Kasch.

\section{co-Kasch modules over commutative rings}

In this section we study co-Kasch modules over commutative rings. We prove that, a Prüfer domain $R$ is Dedekind if and only if every ideal of  $R$ is a  co-Kasch module. We also characterize co-Kasch modules over the ring of integers.

\begin{proposition} Let $R$ be a domain and $M$ be an $R$-module. If $M \neq T(M)$, then $M$ is a co-Kasch module if and only if $\Hom_R(M, \frac{R}{P} )\neq 0$ for each $P \in \Omega .$
\end{proposition}

\begin{proof} Suppose that, $M$ is a co-Kasch module. Since $M \neq T(M),$ there is an $m \in M$ such that $mR \cong R$. Thus every simple module is a homomorphic image of $M$ by the co-Kasch module assumption. Then for each maximal ideal $P$ of $R$, there is an epimorphism $f_P: M\to \frac{R}{P}$. This follows the necessity.  Sufficiency is clear.
\end{proof}

\begin{lemma} \label{lem:ValuationDVR} Let $R$ be a valuation domain. Every ideal of $R$ is a co-Kasch module if and only if $R$ is a DVR.
\end{lemma}

\begin{proof} Sufficiency is clear. To prove the necessity, let $\mathbf{p}$ be the unique maximal ideal of $R$. By the hypothesis $\mathbf{p}$ is a co-Kasch module, and so $\mathbf{p}^2 \neq \mathbf{p}$,  otherwise $\mathbf{p}$ would have no simple factors. Let $p \in \mathbf{p} - \mathbf{p}^2$. Then $\mathbf{p}=Rp$, because $R$ is a valuation domain. Now consider the ideal $I= \cap_{n \in \Z^+} Rp^n.$ Then by  \cite[Proposition 15]{maxprojective} $I$ has no maximal submodules i.e. $\Rad(I)=I$. Thus, since $I$ is co-Kasch by the assumption, $I=0$, by Proposition \ref{prop:radM}(1). Therefore $R$ is a DVR by \cite[page 99, Exercise 4]{Atiyah}.
\end{proof}

%\begin{proposition}Let $R$ be a commutative Noetherian ring and $M=\prod_{P \in \Omega} \frac{R}{P}$. Then every factor of $M$ is co-Kasch if and only if $R$ is semilocal.\end{proposition}

\begin{proposition}  Let $R$ be a Prüfer domain. The following statements are equivalent.
\begin{enumerate}
\item[(1)] Every ideal of $R$ is a co-Kasch module.
\item[(2)]  Every ideal of $R_{\mathbf{p}}$ is a co-Kasch module for every $\mathbf{p} \in \Omega.$
\item[(3)] $R_{\mathbf{p}}$ is a DVR for every $\mathbf{p} \in \Omega.$
\item[(4)] $R$ is Dedekind.
\end{enumerate}
\end{proposition}

\begin{proof} $(1)  \Rightarrow (2)$  Let  $\mathbf{p}$ be a maximal ideal of $R$,  and $B$ be a nonzero proper ideal of $R_{\mathbf{p}}.$ Then there is a proper ideal $A$ of $R$ such that $A_{\mathbf{p}}=B.$ By the hypothesis, $A$ is a co-Kasch module. Thus there is a maximal submodule $C$ of $A$ such that $\frac{A}{C} \cong \frac{R}{\mathbf{p}}.$ Localizing at $\mathbf{p}$ gives: $(\frac{A}{C})_{\mathbf{p}} \cong (\frac{R}{\mathbf{p}})_{\mathbf{p}}$, because localization preserves isomorphisms (see, \cite[page 26: B)]{FC}) Then $\frac{A_{\mathbf{p}}}{C_{\mathbf{p}}} \cong \frac{R_{\mathbf{p}}}{P_{\mathbf{p}}}$ by \cite[page 26: C)]{FC}. Since $R_{\mathbf{p}}$ is a local ring with the unique maximal ideal $\mathbf{p}_{\mathbf{p}}$ and $A_{\mathbf{p}}=B$, we get $\frac{B}{C_{\mathbf{p}}} \cong \frac{R_{\mathbf{p}}}{\mathbf{p}_{\mathbf{p}}}$ is simple. Thus $C_{\mathbf{p}}$ is a maximal ideal of $B$, and so $B$ is a co-Kasch module by Proposition \ref{prop:radM}(b).

$(2)  \Rightarrow (3)$ Since $R$ is Prüfer,  $R_{\mathbf{p}}$ is a valuation domain for each  $\mathbf{p} \in \Omega.$  Thus $(3)$ follows  by Lemma \ref{lem:ValuationDVR}.

$(3)  \Rightarrow (2)$ Every nonzero ideal of a DVR is isomorphic to $R$. Hence $(2)$ follows.

 $(3) \iff (4)$ is clear.

$(4) \Rightarrow (1)$ Let $I$ be a nonzero ideal of $R.$ Note that $I$ is finitely generated. Since $R$ embeds in $I$, $I$ has a subfactor isomorpic to every simple $R$-module. For the proof, we need to show that $\Hom(I, \frac{R}{J}) \neq 0$ for each maximal ideal $J$ of $R$. Assume that $\Hom(I, \frac{R}{J}) = 0$ for some maximal ideal $J$ of $R$. Then $I= J\cdot I$. Localizing at the maximal ideal $J$, we obtain $I_J=J_J \cdot I_J.$ Since $R_J$ is a local ring with the unique maximal ideal $J_J$,the fact that $I_J$ is finitely generated and  $I_J=J_J \cdot I_J$ implies that $I_J=0$ by Nakayama's  Lemma. As $R$ is a domain $I_J=0$ gives $I=0.$ Contradiction. Therefore $I\neq J\cdot I$ for each maximal ideal $J$ of $R$, and so $\Hom(I, \frac{R}{J}) \neq 0$. Hence $I$ is a co-Kasch module, and this proves $(1).$
\end{proof}

Now, we shall give a characterization of co-Kasch modules over the ring of integers. For a $\Z$-module $M$, $T(M)$ is the torsion submodule of $M$. If $p$ is a prime integer the submodule $T_p(M)=\{m \in M \mid p^nm=0\,\, \text{for some positive integer} \, n \}$ is said to be the $p$-primary component of $M$. It is well known that $T(M)=\bigoplus _{p \in \mathcal{P}} T_p(M)$, where $\mathcal{P}$ is the set of prime integers.

\begin{proposition} The following are hold.
\begin{enumerate}
  \item[(a)]  A nonzero torsion $\Z$-module $M$ is a co-Kasch module if and only if  $pM \neq M$ for each prime $p$ with $T_p(M) \neq 0$.
  \item[(b)]  If $T(M) \neq M$, then  $M$ is a co-Kasch module if and only if $pM\neq M$ for each prime $p$.
\end{enumerate}
\end{proposition}

\begin{proof}$(a)$ Since $M$ is torsion, $M= \oplus_{p \in A} T_p(M)$, where $T_p(M)$ is the (nonzero) $p$-primary component of $M$. Suppose $M$ is an $H$-module.  Let $p$ be a prime such that $T_{p}(M)\neq 0$.  Then $S=\Soc (T_p(M)) \cong \frac{\Z}{p\Z}$ is a subfactor of $M$. Since $M$ is a co-Kasch module, $\frac{M}{L} \cong \frac{\Z}{p\Z}$. This implies that $pM\neq M$. This show the necessity. Conversely, let $\frac{X}{Y} \cong \frac{\Z}{p\Z}$ be a simple subfactor of $M$.  We have $\frac{X}{Y} \subseteq  \frac{M}{Y}\cong \oplus \frac{T_p(M)}{T_p(Y)}$. For a prime $q\neq p$, $\Hom(\frac{\Z} {p\Z}, \frac{T_q(M)}{T_q(Y)})=0$. Therefore, $\Hom (\frac{\Z }{p\Z}, \frac{T_p(M)}{T_p(Y)}) \neq 0.$ So that $T_p(M)\neq 0$. Now by the hypothesis $pM \neq M$, and so $M$ has a simple factor  isomorphic to $\frac{\Z}{p\Z}.$

\vspace{0.3cm}
$(b)$ Sufficiency is clear. For necessity suppose $T(M)\neq M.$ Then $M$ has a submodule isomorphic to $\Z$. Thus, $M$ has a factor isomorphic to $\Z_p$ for each prime $p$, by the co-Kasch module assumption.  Therefore $pM \neq M$ for each prime $p$.
\end{proof}

%If $I^2=I$ for some ideal $I$ of a commutative ring $R$, then $PI=I$ for each maximal ideal $P$ containing $I$. In particular, $I$ has no simple factor isomorphic to $\frac{R}{P}.$ This leads to the following.

%\section{Trivial extensions of co-Kaschmodules}

%Motivation of this problem comes from the fact that, over max H-rings, injective modules have this property.

\section*{Acknowledgement}
This work was supported by the Scientific and Technological Research Council of T\"{u}rkiye (TÜBİTAK) (Project number: 122F130).

%Sharpe, D. W., Vamos, P. (1972). Injective Modules. New York, NY: Cambridge University Press.

\end{document}